\documentclass{amsart}

\usepackage{amsmath,amsfonts,amssymb,latexsym,graphics}
\usepackage[bookmarks=false]{hyperref}
\usepackage[dvips]{color}
\usepackage[pdftex]{graphicx}
\usepackage{fullpage}
\usepackage{pstricks}

\DeclareMathOperator{\Tab}{Tab}

\newtheorem{proposition}{Proposition}

\newtheorem{theorem}{Theorem}

\theoremstyle{definition}
\newtheorem{definition}{Definition}
\theoremstyle{remark}
\newtheorem{remark}{Remark}


\title{The Matrix Ansatz, Orthogonal Polynomials, and Permutations}

\author{Sylvie Corteel, Matthieu Josuat-Verg\`es and Lauren K. Williams}

\address{LIAFA, CNRS and Universit\'e Paris-Diderot, Case 7014, 75205 Paris Cedex 13, France}
\address{LRI, CNRS and universit\'e Paris-Sud, 91405 Orsay Cedex, France}
\address{University of California, Berkeley, CA 94720, USA}

\thanks{The three authors are partially supported by the ANR grant
ANR08-JCJC-0011. The third author is also partially supported
by the NSF grant DMS-0854432 and a Sloan Fellowship.}

\email{corteel@liafa.jussieu.fr, josuat@lri.fr, williams@math.berkeley.edu}

\newcommand{\ket}[1]{\ensuremath{|#1\rangle}}
\newcommand{\bra}[1]{\ensuremath{\langle #1|}}
\newcommand{\braket}[2]{\ensuremath{\langle #1|#2 \rangle}}
\newcommand{\Le}{\hbox{\rotatedown{$\Gamma$}}}

\begin{document}

\begin{abstract}
In this paper we outline a {\it Matrix Ansatz} approach to some problems of
combinatorial enumeration.  The idea is that many interesting quantities
can be expressed in terms of products of matrices, where the 
matrices obey certain relations.  We illustrate this approach 
with applications to moments of 
orthogonal polynomials, permutations, signed permutations, and  tableaux.
\end{abstract}

\maketitle 

\begin{center}
\emph{To Dennis Stanton with admiration.}
\end{center}

\bigskip

\section{Introduction}

The aim of this article is to explain a {\it Matrix Ansatz} approach 
to some problems of combinatorial enumeration.  The idea is that many 
interesting enumerative quantities can be expressed in terms of 
products of matrices satisfying certain relations.   
In such a situation,
this Matrix Ansatz can be useful for a variety of reasons:
\begin{itemize}
\item having explicit matrix expressions gives rise to explicit formulas
 for the quantities of interest; 
\item finding combinatorial objects which obey the same relations
  gives rise to a combinatorial formula for the quantities of interest;
\item finding two combinatorial solutions to the same set of relations 
identifies the generating functions for the two 
sets of combinatorial objects.
\end{itemize}

This Matrix Ansatz approach is not 
particularly new and has appeared in various contexts,
notably in statistical physics
(see \cite{HakimNadal}, \cite{Klumper}, \cite{Fannes}, and \cite{DEHP93} 
for a few examples). 
We will take this opportunity
to illustrate its utility for some problems in combinatorics.  

\subsection{The Matrix Ansatz and the ASEP}
The  inspiration
for this article comes from an important paper of 
Derrida, Evans, Hakim, and Pasquier \cite{DEHP93}, in which the authors
described a Matrix Ansatz approach to the stationary
distribution of the {\it asymmetric exclusion process} (ASEP), a model
from statistical physics.  This model can be described as a Markov
chain on $2^n$ states  -- all words of length $n$ in $0$ and $1$ --
where a $1$ in the $i$th position represents a particle in the $i$th
position of a one-dimensional lattice of $n$ sites.  In this Markov chain,
a new particle may enter the lattice at the left with probability 
$\frac{\alpha}{n+1}$,
a particle may exit the lattice to the right with probability $\frac{\beta}{n+1}$, 
and a particle may hop to an empty site to its right or left with probabilities
$\frac{1}{n+1}$ and $\frac{q}{n+1}$ respectively.  See \cite{DEHP93} for more 
details.  A main result of \cite{DEHP93} was the following.

\begin{theorem} \cite{DEHP93}\label{Ansatz}
Suppose that $D$ and $E$ are matrices and $\bra{W}$ and $\ket{V}$ 
are row and column vectors (not necessarily finite-dimensional),
respectively, such that:
\begin{equation} \label{relED}
    DE=qED+D+E, \qquad \alpha\bra{W} E = \bra{W}, 
 \qquad \beta D\ket{V} = \ket{V}, \qquad \braket{W}{V}=1.
\end{equation}
Then the steady state probability that the ASEP with $n$ sites
is in state $(\tau_1,\dots, \tau_n)$ is equal to 
$$\frac{\bra{W} \prod_{i=1}^n (\tau_i D + (1-\tau_i) E) \ket{V}}{Z_n},$$
where $Z_n = \bra{W} (D+E)^n \ket{V}$.
\end{theorem}

Note that one does not lose any information if one forgets the 
denominator $Z_n$, and simply keeps track of the {\it un-normalized
steady state probability} 
$\bra{W} \prod_{i=1}^n (\tau_i D + (1-\tau_i) E) \ket{V}.$

As an example, given $D, E, \bra{W}, \ket{V}$ satisfying the 
conditions of this Matrix Ansatz, 
the un-normalized steady state probability that the ASEP with $5$ sites
is in state $(1, 0, 0, 1, 1)$ is 
$\bra{W} DEEDD \ket{V}$.

There are several remarks worth noting about Theorem \ref{Ansatz}.
First, the theorem does not guarantee {\it a priori} the existence of such matrices
and vectors, nor does it give a method for finding them.  
Second, one may find several very different solutions
to the Matrix Ansatz.  Third, in this case one  does not need to 
find solutions to (\ref{relED}) in order to compute steady state probabilities:
given any word $X$ in $D$ and $E$, one can keep applying the first relation
of (\ref{relED}) to express $X$ as a sum of words of the form 
$E^k D^{\ell}$, then use  the other relations of (\ref{relED}) to evaluate
this sum.  (That being said, explicit solutions to (\ref{relED}) can be useful
for a variety of reasons, as we will see later.)

As we will explain in this article, there are many other interesting quantities
that can be computed via relations similar to those of (\ref{relED}).  
In a previous article \cite{CoWi07a}, 
the first and third authors explored a combinatorial solution to (\ref{relED}),  
finding an explicit solution $D, E, \bra{W}, \ket{V}$ such that expressions 
of the form $\bra{W} X \ket{V}$ enumerate certain tableaux
of a fixed shape corresponding to $X$.  As a consequence, this 
identified steady state probabilities of the ASEP with generating functions
for  tableaux.  In the remainder of the 
introduction, we will explain this example
in more detail.  Subsequent sections will describe variations of the 
Matrix Ansatz and other combinatorial 
enumeration problems which fit into this framework.

\subsection{The Matrix Ansatz and permutation tableaux}

\begin{definition}
A \emph{permutation tableau}  is a Young diagram 
whose boxes are filled with $0$'s and $1$'s such that:
\begin{itemize}
\item each column contains at least one $1$, and 
\item there is no box containing a $0$ which has both a $1$ above it (in the 
same column) and a $1$ to its left (in the same row). 
\end{itemize}
\end{definition}

We always take the English convention for Young diagrams.
Permutation tableaux are a distinguished subset of Postnikov's $\Le$-diagrams
\cite{Pos06}, introduced in connection to the totally non-negative part of the 
Grassmannian.  More specifically, 
if we drop the first condition above, we recover Postnikov's definition
of $\Le$-diagram.  

Note that we allow our Young diagrams to have rows of length $0$, and we define the \emph{length}
of a permutation
tableau to be the sum of its number of rows and columns.
Alternatively, this is the length of the southeast border of its Young diagram.
The permutation tableaux of length $n$ are in bijection
with permutations on $n$ letters \cite{Pos06, StWi07, Bur07, CoNa09}.
See Figure \ref{exa} for two examples of permutation tableaux
of length 8.

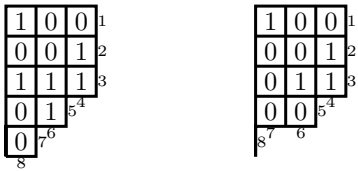
\begin{figure}[h!tp]
\centering
\psset{unit=0.4cm}\psset{linewidth=0.4mm}
\begin{pspicture}(0,0)(3,5)
\psline(0,0)(0,1)(2,1)(2,2)(3,2)(3,5)(0,5)(0,0)
\psline(1,1)(1,5)
\psline(2,2)(2,5)
\psline(0,2)(2,2)
\psline(0,3)(3,3)
\psline(0,4)(3,4)
\psline(0,0)(1,0)(1,1)
\rput(0.5,0.5){$0$}
\rput(0.5,1.5){$0$}
\rput(1.5,1.5){$1$}
\rput(0.5,2.5){$1$}
\rput(1.5,2.5){$1$}
\rput(2.5,2.5){$1$}
\rput(0.5,3.5){$0$}
\rput(1.5,3.5){$0$}
\rput(2.5,3.5){$1$}
\rput(0.5,4.5){$1$}
\rput(1.5,4.5){$0$}
\rput(2.5,4.5){$0$}
\rput(3.2,4.5){\tiny 1}
\rput(3.2,3.5){\tiny 2}
\rput(3.2,2.5){\tiny 3}
\rput(2.5,1.8){\tiny 4}
\rput(2.2,1.5){\tiny 5}
\rput(1.5,0.8){\tiny 6}
\rput(1.2,0.5){\tiny 7}
\rput(0.5,-0.2){\tiny 8}
\end{pspicture}\hspace{2cm}
\begin{pspicture}(0,0)(3,5)
\psline(0,0)(0,1)(2,1)(2,2)(3,2)(3,5)(0,5)(0,0)
\psline(1,1)(1,5)
\psline(2,2)(2,5)
\psline(0,2)(2,2)
\psline(0,3)(3,3)
\psline(0,4)(3,4)
\rput(0.5,1.5){$0$}
\rput(1.5,1.5){$0$}
\rput(0.5,2.5){$0$}
\rput(1.5,2.5){$1$}
\rput(2.5,2.5){$1$}
\rput(0.5,3.5){$0$}
\rput(1.5,3.5){$0$}
\rput(2.5,3.5){$1$}
\rput(0.5,4.5){$1$}
\rput(1.5,4.5){$0$}
\rput(2.5,4.5){$0$}
\rput(3.2,4.5){\tiny 1}
\rput(3.2,3.5){\tiny 2}
\rput(3.2,2.5){\tiny 3}
\rput(2.5,1.8){\tiny 4}
\rput(2.2,1.5){\tiny 5}
\rput(1.5,0.8){\tiny 6}
\rput(0.5,0.8){\tiny 7}
\rput(0.2,0.5){\tiny 8}
\end{pspicture}
\caption{\label{exa} Examples of permutation tableaux}
\end{figure}

Before explaining the connection between the Matrix Ansatz
and permutation tableaux given in \cite{CoWi07a},
we need to define some statistics.
A $0$ in a permutation tableau
is called {\it restricted} if it has a $1$ above it in the same column.
An {\it unrestricted row} is a row that does
not contain any restricted $0$.
A {\it topmost} $1$ is a $1$ which is topmost in its column, and 
a {\it superfluous} $1$ 
is a $1$ that is not topmost. 
If $T$ is a permutation tableau, 
let $wt(T)$ be the number of superfluous $1$'s, let $f(T)$ be the 
number of $1$'s
in the first row, and let $u(T)$ be 
the number of unrestricted rows minus $1$.

Label the southeast border of a permutation tableaux
with $D$'s and $E$'s going from North-East to South-West such that each
South step is labelled by $D$ and each West step is labelled by $E$.
As the first step of the border is always south, we can ignore it,
and encode
the border of a permutation tableau of length $n+1$ with a 
word in the alphabet $\{D,E\}$ of length $n$. The shape of a tableau of length $n+1$
is the corresponding
word in $\{D,E\}^n$. Each corner
of the tableau corresponds to a pattern $DE$ in the word.
For example, the word associated to the diagram at the left of Figure
\ref{decompA} is $DDEDEED$.

If $X$ is a word in $D$ and $E$,
let \[
\Tab(X) =\sum_{T} q^{wt(T)}\alpha^{-f(T)}\beta^{-u(T)},
\]
where the sum is over all permutation tableaux $T$ of shape $X$.

We first claim that for any words $X$ and $Y$ in $D$ and $E$, 
\begin{equation} \label{firstrelation}
\Tab(XDEY) = q \Tab(XEDY) + \Tab(XDY) + \Tab(XEY). 
\end{equation}
To see that this is true, look at the content of a corner box
of a tableau of shape $XDEY$.
It can contain: 
\begin{itemize}
\item a superfluous $1$, in which case the box can be deleted, leaving
a permutation tableau of shape $XEDY$.
\item a topmost $1$, in which case its  column may be deleted,
leaving a permutation tableau of shape $XDY$.   
\item a $0$ (necessarily with only $0$'s to its left), in which case its entire
row may be deleted, leaving a permutation tableau of shape $XEY$.
\end{itemize}
See Figure \ref{decompA}.  In all cases, the monomial associated to the smaller
permutation tableau is easily described in terms of the monomial associated to 
the original tableau, yielding (\ref{firstrelation}).

\begin{figure}[htp]
\centering
\psset{unit=0.4cm}\psset{linewidth=0.4mm}
\begin{pspicture}(0,-1)(3,6)
\psline(0,0)(0,1)(2,1)(2,2)(3,2)(3,5)(0,5)(0,0)
\psline(1,1)(1,5)
\psline(2,2)(2,5)
\psline(0,2)(2,2)
\psline(0,3)(3,3)
\psline(0,4)(3,4)
\rput(3.2,3.5){\tiny $D$}
\rput(3.2,2.5){\tiny $D$}
\rput(2.5,1.8){\tiny $E$}
\rput(2.2,1.5){\tiny $D$}
\rput(1.5,0.8){\tiny $E$}
\rput(0.5,0.8){\tiny $E$}
\rput(0.2,0.5){\tiny $D$}
\rput(1,-1){\small{$Tab(XDEY)$}}
\end{pspicture}\hspace{.5cm}
\begin{pspicture}(-1,-1)(3,6)
\rput(-1,3){$=$}
\psline(0,0)(0,1)(2,1)(2,2)(3,2)(3,5)(0,5)(0,0)
\psline(1,1)(1,5)
\psline(2,2)(2,5)
\psline(0,2)(2,2)
\psline(0,3)(3,3)
\psline(0,4)(3,4)
\rput(3.2,3.5){\tiny $D$}
\rput(3.2,2.5){\tiny $D$}
\rput(2.5,1.8){\tiny $E$}
\rput(2.2,1.5){\tiny $D$}
\rput(1.5,0.8){\tiny $E$}
\rput(0.5,0.8){\tiny $E$}
\rput(0.2,0.5){\tiny $D$}
\rput(2.5,2.5){$1$}
\rput(1,-1){\small{$=q\Tab(XEDY)$}}
\end{pspicture}
\hspace{.3cm}
\begin{pspicture}(-1,-1)(3,6)
\rput(-1,3){+}
\psline(0,0)(0,1)(2,1)(2,2)(3,2)(3,5)(0,5)(0,0)
\psline(1,1)(1,5)
\psline(2,2)(2,5)
\psline(0,2)(2,2)
\psline(0,3)(3,3)
\psline(0,4)(3,4)
\rput(3.2,3.5){\tiny $D$}
\rput(3.2,2.5){\tiny $D$}
\rput(2.5,1.8){\tiny $E$}
\rput(2.2,1.5){\tiny $D$}
\rput(1.5,0.8){\tiny $E$}
\rput(0.5,0.8){\tiny $E$}
\rput(0.2,0.5){\tiny $D$}
\rput(2.5,2.5){$1$}
\rput(2.5,3.5){$0$}
\rput(2.5,4.5){$0$}
\rput(1,-1){\small{$+\Tab(XDY)$}}
\end{pspicture}
\hspace{.3cm}
\begin{pspicture}(-1,-1)(3,6)
\rput(-1,3){+}
\psline(0,0)(0,1)(2,1)(2,2)(3,2)(3,5)(0,5)(0,0)
\psline(1,1)(1,5)
\psline(2,2)(2,5)
\psline(0,2)(2,2)
\psline(0,3)(3,3)
\psline(0,4)(3,4)
\rput(3.2,3.5){\tiny $D$}
\rput(3.2,2.5){\tiny $D$}
\rput(2.5,1.8){\tiny $E$}
\rput(2.2,1.5){\tiny $D$}
\rput(1.5,0.8){\tiny $E$}
\rput(0.5,0.8){\tiny $E$}
\rput(0.2,0.5){\tiny $D$}
\rput(2.5,2.5){$0$}
\rput(1.5,2.5){$0$}
\rput(0.5,2.5){$0$}
\rput(1,-1){\small{$+\Tab(XEY)$}}
\end{pspicture}
\caption{\label{decompA} Decomposition of a permutation tableau}
\end{figure}

We also claim that for any word $X$ in $D$ and $E$,
\begin{equation} \label{secondthirdrelation}
\alpha \Tab(EX) = \Tab(X) \text{ and } \beta \Tab(XD) = \Tab(X).
\end{equation}

These relations are clear upon inspection.  Comparing 
equations (\ref{firstrelation}) and (\ref{secondthirdrelation}) to 
the equations of the Matrix Ansatz (\ref{Ansatz}), we see that 
the generating functions $\Tab(X)$ for permutation tableaux of 
shape $X$ obey the same recursions as do the un-normalized
steady state probabilities of 
the ASEP. Hence for each word $X$ of length $n$ in $D$ and $E$,
$\Tab(X)$ computes the (un-normalized) steady state probability
of being in the corresponding state of the ASEP with $n$ sites.


Alternatively, one could concretely define the following matrices
and vectors, with rows and columns indexed by the non-negative integers.  Let
$D$  be the (infinite) upper triangular
matrix $(D_{i,j})_{i,j \geq 0}$ such that
$$D_{i,i+1}=\beta^{-1},\ \qquad D_{i,j}=0\qquad {\rm otherwise}.
$$
Let $E$  be the (infinite) lower triangular matrix
$(E_{ij})_{i,j\geq 0}$ such that for $j \leq i$,
$$E_{ij} = \beta^{i-j} \big(\alpha^{-1} q^{j} \binom ij +
\sum_{r=0}^{j-1} {i-j+r \choose r} q^r\big).$$
Otherwise, $E_{ij} = 0$.
Also let 
$\bra{W}$ be the (row) vector $(1,0,0,\dots )$ and $\ket{V}$ be the (column)
vector
$(1,1,1,\dots)^T$.
Then one can check that these matrices and vectors satisfy (\ref{Ansatz}), 
and 
hence that $\bra{W} X \ket{V}$ is equal to $\Tab(X)$.  More specifically,
one can think of $D$ and $E$ as operators acting on the infinite-dimensional
vector space indexed by all permutation tableaux (of all shapes and lengths),
where $D$ acts by adding a new row of length $0$ and $E$ acts by adding
a new column at the left of an existing tableau, in all possible ways.
See \cite{CoWi07a} for details.

The upshot of both arguments is that the steady state probability that the ASEP
is in state $\tau$ is proportional to the generating function of permutation 
tableaux of shape $\prod_{i=1}^n (\tau_iD+(1-\tau_i)E$.  
Additionally, the second
argument provides an explicit formula for this generating function, as 
a matrix product.

We remark that there is a more general version of the ASEP in which
particles enter and exit at both sides of the lattice, with probabilities
$\alpha, \beta, \gamma,$ and $\delta$, and there is also a tableaux
formula for the stationary distribution of this more general model
\cite{CoWi10}.  However, this formula is considerably harder
to prove than the special case described above;
the methods described here are not sufficient.

\subsection{The Matrix Ansatz and Motzkin paths}

Another solution to the Matrix Ansatz of Theorem \ref{Ansatz} was given 
in \cite{BCEPR,BECE}, and is as follows.
Let $\tilde\alpha=(1-q)\frac 1\alpha-1$ and $\tilde\beta=(1-q)\frac 1\beta-1$.  Define 
vectors $\bra{W}=(1,0,0,\dots)$ and $\ket{V} = (1,0,0,\dots)^T$, and
matrices $D=(D_{i,j})_{i,j\geq0}$
and $E=(E_{i,j})_{i,j\geq0}$
where all entries are equal to $0$ except
\begin{equation}
(1-q)D_{i,i}=1-\tilde\beta q^i,\qquad
(1-q)D_{i,i+1}= 1-\tilde\alpha\tilde\beta q^i,\qquad
(1-q)E_{i,i}=1-\tilde\alpha q^i,\qquad (1-q)E_{i+1,i}= 1-  q^{i+1}.
\label{useafter}
\end{equation}

Recall that a {\it Motzkin path} is a lattice path in the plane 
with steps east $(1,0)$, northeast $(1,1)$, and southeast $(1,1)$,
which starts at the origin, 
always stays at or above
the $x$-axis, and 
ends at the $x$-axis. 
One often puts weights on the steps of a Motzkin path
and then considers the generating function for all Motzkin paths of 
a fixed length.

Note that if $\bra{W}=(1,0,\ldots )$, $\ket{V}
=(1,0,\ldots)^T$,  and $D$ and $E$ are tridiagonal matrices (that is,
$D_{i,j}=E_{i,j}=0$ for $|i-j|>1$), then 
for any word $X$ in $D$ and $E$, the quantity 
$\bra{W}X\ket{V}$ is the generating function for all
weighted Motzkin paths of length $n$ and {\it shape} $X=X_1\ldots X_n$. 
By {\it shape} $X$, we mean that if $X_j=D$ then the $j$th step of 
the Motzkin path must be a northeast or east step, and if $X_j=E$, then
the $j$th step of the Motzkin path must be a southeast or east step.
Furthermore, if the $j$th step of the Motzkin path
starts at height $i$, its weight is:
\begin{itemize}
\item $D_{i,i+1}$ if the step is northeast and $X_j=D$
\item $E_{i,i+1}$ if the step is northeast and $X_j=E$
\item $D_{i,i}$ if the step is east and $X_j=D$
\item $E_{i,i}$ if the step is east and $X_j=E$
\item $D_{i,i-1}$ if the step is southeast and $X_j=D$
\item $E_{i,i-1}$ if the step is southeast and $X_j=E$
\end{itemize}

Moreover the quantity $Z_n = \bra{W} (D+E)^n \ket{V}$ is the generating function
for all weighted Motzkin paths of length $n$ where the weight of
a northeast (resp. east, southeast) step starting at height 
$i$ is 
$D_{i,i+1}+E_{i,i+1}$, (resp. $D_{i,i}+E_{i,i}$, $D_{i,i-1}+E_{i,i-1}$).

It follows that the above solution $D, E, \bra{W}, \ket{V}$
to the Matrix Ansatz identifies steady state probabilities of 
the ASEP with generating functions for weighted Motzkin paths
of the corresponding shape.  

Additionally, these Motzkin paths provide a link to orthogonal polynomials,
via the combinatorial theory of orthogonal polynomials provided by 
\cite{Fla82,Vie88}.  More specifically, 
the moments of the weight function
of a family of monic orthogonal polynomials can be identified with 
the generating functions for weighted Motzkin paths,
where the weights on the steps of the 
Motzkin paths come directly from the coefficients of the 
three-term recurrence for the orthogonal polynomials.  This will be explained
more carefully in the following section.  
This allows us to interpret the partition
function $Z_n=\bra{W} (D+E)^n \ket{V}$  of the ASEP with $n$ sites as the 
$n$th moment of a family of orthogonal polynomials.
Indeed, the solution of the Matrix Ansatz given here corresponds to the
moments of the Al-Salam-Chihara polynomials. The link
between the ASEP and orthogonal polynomials was first made in \cite{Sa99}.
Most of the time,  when we define orthogonal
polynomials, we will refer to  the survey paper
\cite{KoSw98}. All the original articles are cited in \cite{KoSw98}.


This article is organized as follows. In Section \ref{sec:orth}, we recall some 
elementary facts concerning moments of orthogonal polynomials, and explain how
the Matrix Ansatz sheds new light on some known examples, {\it e.g.}
the connection between rook placements and Hermite polynomials.
 In Section \ref{sec:sign},
we apply this method to the enumeration of permutations,
signed permutations, and type B permutation tableaux 
\cite{LaWi08}, and we make a link with two different kinds of $q$-Laguerre polynomials.
Finally, in Section \ref{sec:gen}, we show that the Matrix Ansatz gives rise to a
new combinatorial interpretation of a generalization of Genocchi numbers
defined by Dumont and Foata \cite{DuFo76}.

\textsc{Acknowledgements:}
The third author is grateful to Ira Gessel for a useful conversation 
about orthogonal polynomials.

\section{Orthogonal polynomials}
\label{sec:orth}

In this section we review some notions about orthogonal polynomials, in particular
a combinatorial interpretation of their moments.
We then revisit some classical results about orthogonal polynomials,
illustrating the Matrix Ansatz approach.
Note that throughout this section,
we take $\bra{W}=(1,0,0,\dots)$ and $\ket{V}=\bra{W}^T$.

\begin{definition}
We say $\{P_k(x)\}_{k \geq 0}$ is a family of \emph{orthogonal polynomials} if 
there exists a linear functional $f: K[x]\to K$ such that:
\begin{itemize}
\item $\deg(P_k) = k$ for $k \geq 0$,
\item $f(P_k P_{\ell}) = 0$ if $k \neq \ell$,
\item $f(P_k^2) \neq 0$ for $k \geq 0$.
\end{itemize}
The \emph{$n$th moment} of $\{P_k(x)\}_k$  is defined 
to be $\mu_n=f(x^n)$.
\end{definition}

It might seem that $\mu_n$ depends on $f$ and not just on $\{P_k(x)\}_{k \geq 0}$. But since
$f(P_n)=0$ for any $n\geq1$, we can obtain the moments recursively from the initial
moment $\mu_0$.  By convention we always take $\mu_0=1$.

By Favard's Theorem, monic orthogonal polynomials satisfy a three-term recurrence.
\begin{theorem}
Let $\{P_k(x)\}_{k \geq 0}$ be a family of monic orthogonal polynomials.  Then 
there exist coefficients $\{b_k\}_{k \geq 0}$ and $\{\lambda_k\}_{k \geq 1}$
in $K$ such that 
$P_{k+1}(x) = (x-b_k) P_k(x) - \lambda_k P_{k-1}(x)$.
\end{theorem}

By work of \cite{Fla82,Vie88}, the $n$th moment of a family
of monic orthogonal polynomials can be identified with the 
generating function for weighted Motzkin paths as follows.

\begin{theorem}\cite{Fla82, Vie88}\label{moment-motzkin}
Consider a family of monic orthogonal polynomials $\{P_k(x)\}_{k \geq 0}$ 
which satisfy the three-term recurrence
$P_{k+1}(x) = (x-b_k) P_k(x) - \lambda_k P_{k-1}(x)$, for 
$\{b_k\}_{k \geq 0}$ and $\{\lambda_k\}_{k \geq 1}$
in $K$.  Then the $n$th moment $\mu_n$ of $\{P_k(x)\}_{k \geq 0}$ is equal to 
$\bra{W} M^n \ket{V}$, where 
$M=(m_{ij})_{i,j\geq 0}$ is the tridiagonal matrix with
rows and columns indexed by the non-negative integers, such that
$m_{i,i-1}=\lambda_i,$ $m_{ii}=b_i$, and $m_{i,i+1}=1$.
Equivalently, $\mu_n$ is equal to the generating
function for weighted Motzkin paths of length $n$, where the 
northeast steps have weight $1$,
the east steps at height $i$ have weight $b_i$,
and the southeast steps starting at height $i$ have weight $\lambda_i$.
\end{theorem}

The third author would like to thank Ira Gessel for explaining the following
simple proof to her.

\begin{proof}
Let us re-write the three-term recurrence
as $xP_k(x) = P_{k+1}(x) +b_k P_k(x) + \lambda_k P_{k-1}(x)$. 
Let $P=(P_0(x), P_1(x), P_2(x), \dots )^T$ be the vector whose $i$th component
is $P_i(x)$.
Then the fact that the $P_k(x)$ satisfy the three-term recurrence is equivalent
to the statement that $MP = xP$.

It follows that $M^n P = x^n P$. Let $f$ be the linear functional
associated to $\{P_k(x)\}$.  Applying $f$
 to both sides gives 
\begin{equation} \label{equality}
M^n f(P) = f(x^n P).
\end{equation}
We now analyze the first entry of the vector on either side of 
(\ref{equality}).  
Note that $P_0=1$ so $f(x^n P_0)=f(x^n)=\mu_n$.
And $f(P_0)=1$, so the first entry of the left-hand-side of 
(\ref{equality}) is equal to $\bra{W} M^n \ket{V}$.
It is easy to see that this is equal to the generating function for Motzkin paths
of length $n$, where the northeast steps are weighted $1$,
the east steps at height $i$ are weighted $b_i$,
and the southeast steps starting at height $i$ are weighted $\lambda_i$.
This completes the proof.
\end{proof}

The following remarks will be useful in subsequent sections.
\begin{remark}
Given a tridiagonal matrix $M$, note that if we define
another tridiagonal matrix $M'=(m'_{ij})$ such that 
$m'_{ii} = m_{ii}$ and $m'_{i, i+1} m'_{i+1, i} = m_{i, i+1} m_{i+1, i}$,
then $\bra{W} M^n \ket{V} = \bra{W} M'^n \ket{V}.$
\end{remark}

\begin{remark} \label{shift}
If $\{P_k(x)\}_{k\geq0}$ is an orthogonal sequence with moments $\{\mu_n\}_{n\geq0}$, 
the shifted sequence $\{P_k(ax+b)\}_{k\geq0}$ with $a,b\in K$ and $a\neq 0$ is also
orthogonal, and its $n$th moment is $\frac{1}{a^n}\sum_{k=0}^n \binom n k \mu_k
(-b)^{n-k} $. In particular, if $\mu_n$ has the form $\mu_n=\bra{W}M^n\ket{V}$, then
the moments of the shifted sequence are $\bra{W}M'^n\ket{V}$ where $M'=\frac1a(M-bI)$.
\end{remark}

\subsection{Rook placements and $q$-Hermite polynomials}
Let $[n]_q=1+q+\ldots +q^{n-1}$. One of the simplest classes of orthogonal
polynomials are the ``continuous big $q$-Hermite polynomials'', defined by the recurrence
\begin{equation}
 x H_n(x|q) =  H_{n+1}(x|q) + aq^nH_n(x|q) + [n]_q H_{n-1}(x|q).
\end{equation}
When $a=0$ they specialize to the ``continuous $q$-Hermite polynomials'', and in this case 
a combinatorial interpretation of the moments was given in \cite{ISV87} in terms of perfect 
matchings and crossings. From Theorem \ref{moment-motzkin},
the moments are $\mu^h_n=\bra{W}M^n\ket{V}$, where the matrix $M$ has coefficients
\begin{equation}
 m_{i+1,i}=1, \qquad m_{i,i}=aq^i, \qquad m_{i,i+1}=[i]_q, \qquad m_{i,j} = 0\quad \hbox{if } |i-j|>1.
\end{equation}
We can write $M=F+U$ where $F$ is strictly upper triangular with $0$'s on 
the diagonal, and $U$ is lower 
triangular. We can check that 
\begin{equation}\label{FUAnsatz}
    FU-qUF = I, \qquad \bra{W} U = a\bra{W}, \qquad F\ket{V} = 0.
\end{equation}
So we have a ``Matrix Ansatz" for the moments $\mu^h_n$. 
Indeed, just as permutation tableaux are described by the Matrix Ansatz of 
(\ref{relED}), we can search for tableaux which are enumerated by products
$\bra{W} X \ket{V}$, where $X$ is a word in $F$ and $U$, and $F$ and $U$ satisfy
the relations of (\ref{FUAnsatz}).  It's easy to see that 
$\bra{W} X \ket{V}$ enumerates rook placements in Young diagrams of ``shape" $X$,
where 
\begin{itemize}
\item there is exactly one rook per row,
\item $a$ counts the columns without rook,
\item $q$ counts each cell with no rook to its right in the same row, nor below it in the
same column.
\end{itemize}
And then since the moments $\mu^h_n$ of the continuous big $q$-Hermite polynomials are given by
$\bra{W} (F+U)^n \ket{V}$, we recover results of \cite{Var05}: 
$\mu^h_n$ enumerates such rook placements in all Young diagrams
where the number of rows plus the number of columns is $n$.

There is a simple bijection between these rook placements and involutions on $[n]$.  
To obtain the involution from
the rook placement, we label the steps in the southeast border of the Young diagram, and 
draw some arches as follows: for each rook lying in 
column $i$ and row $j$,
we draw an arch joining $i$ and $j$. 
See Figure \ref{fig_inv}; the rooks are represented by 
$\circ$.  Under this bijection, $a$ 
counts the fixed points, and $q$ keeps track of  a statistic called {\it total crossing number}
defined in \cite{MSS07}.
This extends the case $a=0$ given by Ismail, Stanton, and Viennot \cite{ISV87}.

We can also consider the Motzkin paths corresponding to 
the product $\mu^h_n=\bra{W}(F+U)^n\ket{V}$. There is a simple bijection between these
paths and involutions. This was given in \cite{Pen95} in the case $a=0$, and in the present case
with $a$ general, Penaud's construction can be adapted so that horizontal steps in the Motzkin
paths correspond to fixed points in the involution. See Figure \ref{fig_inv} for a rook placement,
the corresponding involution, and the corresponding Motzkin path. 
The weighted Motzkin path is obtained 
from the involution as follows: the $i$th step is $\nearrow$ if $i$ is the smallest element of an
arch, it is $\rightarrow$ with weight $aq^k$ if $i$ is a fixed point with $k$ arches above it, 
and it is a step $\searrow$ with weight $q^k$ if $i$ is the largest element of an arch $(j,i)$ and 
there are $k$ arches $(u,v)$ such that $j<u<i<v$.

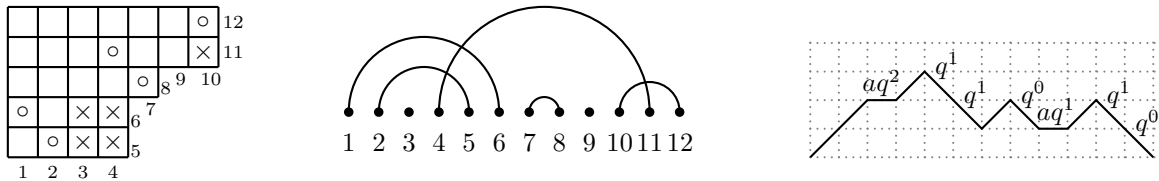
\begin{figure}[h!tp]  \center \psset{unit=4mm}
\begin{pspicture}(1,0)(7,5)
\psline(0,5)(0,0)\psline(0,5)(7,5)
\psline(1,5)(1,0)\psline(0,4)(7,4)
\psline(2,5)(2,0)\psline(0,3)(7,3)
\psline(3,5)(3,0)\psline(0,2)(5,2)
\psline(4,5)(4,0)\psline(0,1)(4,1)
\psline(5,5)(5,2)\psline(0,0)(4,0)
\psline(6,5)(6,3)
\psline(7,5)(7,3)
\rput(0.5,1.5){$\circ$}\rput(1.5,0.5){$\circ$}\rput(3.5,3.5){$\circ$}
\rput(4.5,2.5){$\circ$}\rput(6.5,4.5){$\circ$}\rput(2.5,1.5){$\times$}
\rput(2.5,0.5){$\times$}\rput(3.5,1.5){$\times$}\rput(3.5,0.5){$\times$}
\rput(6.5,3.5){$\times$}
  \rput(0.5,-0.5){\scriptsize 1}
  \rput(1.5,-0.5){\scriptsize 2}
  \rput(2.5,-0.5){\scriptsize 3}
  \rput(3.5,-0.5){\scriptsize 4}
  \rput(4.25,0.2){\scriptsize 5}
  \rput(4.25,1.2){\scriptsize 6}
  \rput(4.75,1.6){\scriptsize 7}
  \rput(5.25,2.3){\scriptsize 8}
  \rput(5.75,2.6){\scriptsize 9}
  \rput(6.75,2.6){\scriptsize 10}
  \rput(7.5,3.5){\scriptsize 11}
  \rput(7.5,4.5){\scriptsize 12}  
\end{pspicture}
\hspace{1.5cm} \psset{unit=4mm}
\begin{pspicture}(1,-1.5)(12,3)
\psdots(1,0)(2,0)(3,0)(4,0)(5,0)(6,0)(7,0)(8,0)(9,0)(10,0)(11,0)(12,0)
  \rput(1,-1){1}
  \rput(2,-1){2}
  \rput(3,-1){3}
  \rput(4,-1){4}
  \rput(5,-1){5}
  \rput(6,-1){6}
  \rput(7,-1){7}
  \rput(8,-1){8}
  \rput(9,-1){9}
  \rput(10,-1){10}
  \rput(11,-1){11}
  \rput(12,-1){12}
\psarc(3.5,0){2.5}{0}{180}\psarc(3.5,0){1.5}{0}{180}\psarc(7.5,0){3.5}{0}{180}\psarc(7.5,0){0.5}{0}{180}
\psarc(11,0){1}{0}{180}
\end{pspicture}
\hspace{1.5cm} \psset{unit=3.8mm}
\begin{pspicture}(10,4)
\psgrid[gridcolor=gray,griddots=4,subgriddiv=0,gridlabels=0](0,0)(12,4)
\psline(0,0)(1,1)(2,2)(3,2)(4,3)(5,2)(6,1)(7,2)(8,1)(9,1)(10,2)(11,1)(12,0)
\rput(2.5,2.6) {$aq^2$}
\rput(4.8,3.2) {$q^1$}
\rput(5.8,2.2) {$q^1$}
\rput(7.8,2.2) {$q^0$}
\rput(8.6,1.6) {$aq^1$}
\rput(10.8,2.2) {$q^1$}
\rput(11.8,1.2){$q^0$}
\end{pspicture}
\caption{ \label{fig_inv}
A rook placement, the corresponding involution and weighted Motzkin path.
}
\end{figure}

\subsection{$0-1$ Tableaux and a $q$-analogue of Charlier polynomials}

The Stirling numbers appear as moments of Charlier polynomials. We consider here the $q$-Charlier
polynomials defined by de M\'edicis, Stanton, and White in \cite{MSW95} as a rescaled version of 
Al-Salam-Carlitz polynomials, with the recurrence relation
\begin{equation}
      xC_n(x, a; q) = C_{n+1}(x, a; q) 
        +  ( aq^n + \left[n\right]_q ) C_n (x, a; q) + a\left[n\right]_q q^{n-1} C_{n-1} (x, a; q).
\end{equation}
From the same reference we know that the $n$th moment of this sequence is $\sum_{k=1}^na^kS_q(n,k)$
where $S_q(n,k)$ is the Carlitz $q$-Stirling number, satisfying the relation
\begin{equation}
   S_q(n,k) = S_q(n-1,k-1) + [k]_q S_q(n-1,k).
\end{equation}
A combinatorial interpretation of $S_q(n,k)$ was given by Leroux in \cite{Ler90}. A {\it $0-1$  tableau} is
a filling of a Young diagram with $0$'s and $1$'s 
such that there is exactly one
 $1$ in each column. Then 
$S_q(n,k)$ is the generating function for $0-1$ 
tableaux with $n-k$ columns and $k$ rows, where $q$ counts
the number of $0$'s above a $1$.

Let us explain how we can reproduce Leroux's combinatorial interpretation
via the Matrix Ansatz method.
From the recurrence relation, the moments of the 
$q$-Charlier polynomials are $\mu^c_n=\bra{W}M^n\ket{V}$ 
where the tridiagonal matrix $M$ has coefficients
\begin{equation}
m_{i+1,i} = [i+1]_q, \qquad m_{i,i} = aq^i + [i]_q, \qquad m_{i,i+1} = aq^i, \qquad m_{i,j} = 0\quad \hbox{if } |i-j|>1.
\end{equation}
We can write $M=aX+Y$ where $X$ and $Y$ do not depend on $a$. These matrices
satisfy the relations
\begin{equation} \label{XY}
   XY-qYX = X, \qquad  \bra{W} Y = 0,  \qquad  X \ket{V} = \ket{V}.
\end{equation}
The Matrix Ansatz method shows that with these relations, $\bra{W}(aX+Y)^n\ket{V}$ is 
a generating function of $0-1$ tableaux. Indeed there is a recursive decomposition of $0-1$ tableaux  
associated with the relation (\ref{XY}), in the same way that permutation tableaux are associated
with the relations (\ref{relED}). So the method explains why these tableaux appear as moments of 
the $q$-Charlier polynomials.

\subsection{Rook placements and another $q$-analogue of Charlier polynomials}

Another $q$-analogue of Charlier polynomials was defined by Kim, Stanton, and Zeng in \cite{KSZ06}
as a rescaled version of Al-Salam-Chihara polynomials. The recurrence is:
\begin{equation} \label{qchar_rec}
  x C^*_{n}(x,a; q) = C^*_{n+1}(x,a; q) +
     (a+\left[n\right]_q)  C^*_{n}(x,a; q)
     + a \left[n\right]_q  C^*_{n-1}(x,a; q).
\end{equation}
The $n$th moment is $\mu_n^{c^*}=\bra{W}M\ket{V}$, where $M$ has coefficients:
\begin{equation} \label{def_Mqchar}
   m_{i+1,i} = a, \qquad m_{i,i} = a + [i]_q, \qquad m_{i,i+1}  = [i+1]_q, \qquad m_{i,j} = 0\quad \hbox{if } |i-j|>1.
\end{equation}
Let $F=(f_{i,j})_{i,j\geq0}$ and $U=(u_{i,j})_{i,j\geq0}$ be such that $f_{i,i+1}=[i+1]_q$ and 
$u_{i+1,i}=1$, all other coefficients being 0.
We can check that  $M=(U+I)(F+aI)$ where $F$ and $U$ are the matrices defined above (for the 
$q$-Hermite polynomials ) satisfying $FU-qUF=I$. From $\mu_n^{c^*}=\bra{W}((U+I)(F+aI))^n\ket{V}$, 
the Matrix Ansatz method gives a combinatorial interpretation of the moments in terms of rook placements. 
Indeed, in the expansion of $((U+I)(F+aI))^n$, each term $m$ corresponds to the choice of some columns
and rows in the staircase Young diagram, and $\bra{W}m\ket{V}$ counts the rook placements with no free
row or column of shape $m$. So we can see that $\mu_n^{c^*}$ counts the rook placements in the staircase
Young diagram of length $2n$, possibly with empty rows and columns, and $q$ counts the number
of inversions, {\it i.e.} cells having a rook below and a rook to its left. 

A simple bijection links
these rook placements with set partitions so that $q$ counts the number of crossings and $a$
counts the number of blocks.  This result was first given in \cite{KSZ06}. Motzkin paths in this 
case are known as {\it Charlier diagrams} \cite{KaZe06}. See Figure \ref{rook_stair} for an example of
a rook placement, the corresponding set partition $\pi = (1,5,7)(2,6)(3,4)$, and the Charlier diagram. 
The set partition is obtained from the rook placement as follows: we label the inner corner of the 
Young diagrams with integers, and for any rook at the intersection of a column with label $i$ and
a row with label $j$ we draw an arch from $i$ to $j$. Each block of the set partition is given by
a sequence of chained arches. The weighted Motzkin path is obtained from the set partition in the 
following way. The $i$th step is $\nearrow$ with weight $a$ if $i$ is the
minimum element of a non-singleton block, it is $\rightarrow$ with weight $a$ if $i$ is a singleton
block, it is $\rightarrow$ with weight $q^k$ if $i$ is the minimal element of an arch and the maximal 
element of another arch $(j,i)$, and it is $\searrow$ with weight $q^k$ if $i$ is the maximal element
of an arch $(j,i)$ but not the minimal element of another arch. Here $k$ is the number of arches $(u,v)$ 
such that $j<u<i<v$.

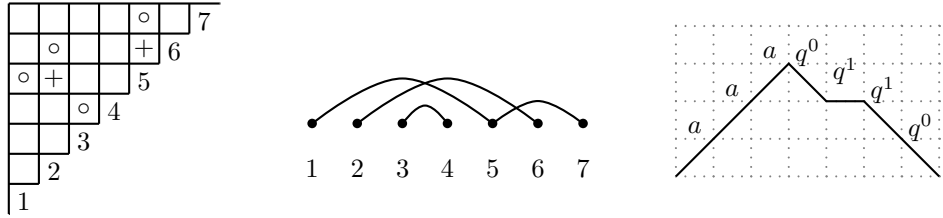
\begin{figure}[h!tp] \center \psset{unit=4mm}
\begin{pspicture}(1,0)(7,7)
\psline(0,7)(0,0)  \psline(0,7)(7,7)
\psline(1,7)(1,1)  \psline(0,6)(6,6)
\psline(2,7)(2,2)  \psline(0,5)(5,5)
\psline(3,7)(3,3)  \psline(0,4)(4,4)
\psline(4,7)(4,4)  \psline(0,3)(3,3)
\psline(5,7)(5,5)  \psline(0,2)(2,2)
\psline(6,7)(6,6)  \psline(0,1)(1,1)
\psline(7,7)(7,7)  \psline(0,0)(0,0)
\rput(0.5,4.5){$\circ$}\rput(1.5,5.5){$\circ$}
\rput(2.5,3.5){$\circ$}\rput(4.5,6.5){$\circ$}
\rput(1.5,4.5){$+$}    \rput(4.5,5.5){$+$}
\rput(0.5,0.4){ 1 }
\rput(1.5,1.4){ 2 }
\rput(2.5,2.4){ 3 }
\rput(3.5,3.4){ 4 }
\rput(4.5,4.4){ 5 }
\rput(5.5,5.4){ 6 }
\rput(6.5,6.4){ 7 }
\end{pspicture}
\hspace{1cm}   \psset{unit=6mm}
\begin{pspicture}(1,-2)(7,1)
\psdots(1,0)(2,0)(3,0)(4,0)(5,0)(6,0)(7,0)
\pscurve[arrowscale=1.5]{-}(1,0)(3,1)(5,0)     \pscurve[arrowscale=1.5]{-}(5,0)(6,0.5)(7,0)
\pscurve[arrowscale=1.5]{-}(3,0)(3.5,0.4)(4,0) \pscurve[arrowscale=1.5]{-}(2,0)(4,1)(6,0)
\rput(1,-1){1}
\rput(2,-1){2}
\rput(3,-1){3}
\rput(4,-1){4}
\rput(5,-1){5}
\rput(6,-1){6}
\rput(7,-1){7}
\end{pspicture}
\hspace{1cm}   \psset{unit=5mm}
\begin{pspicture}(0,-1)(7,4)
\psgrid[gridcolor=gray,griddots=4,subgriddiv=0,gridlabels=0](0,0)(7,4)
\psline(0,0)(1,1)(2,2)(3,3)(4,2)(5,2)(6,1)(7,0)
\rput(0.5,1.3){$a$}
\rput(1.5,2.3){$a$}
\rput(2.5,3.3){$a$}
\rput(3.5,3.3){$q^0$}
\rput(4.5,2.8){$q^1$}
\rput(5.5,2.3){$q^1$}
\rput(6.5,1.3){$q^0$}
\end{pspicture}
\caption{\label{rook_stair}
A rook placement in a staircase diagram, a set partition, and a Charlier diagram.
}
\end{figure}

\section{Permutations and signed permutations}
\label{sec:sign}

\subsection{Permutations and inversions}

De M\'edicis and Viennot \cite{MeVi94} studied a $q$-analogue of Laguerre polynomials 
(with one parameter set to $1$),
whose $n$th moment is the classical $q$-factorial $[n]_q!$.
They are a particular case of the
little $q$-Jacobi polynomials. The recurrence relation is: 
\begin{equation}
  x L_n(x) = L_{n+1}(x) + q^n([n]_q+[n+1]_q) L_n(x) + q^{2n-1}[n]_q^2 L_{n-1}(x).
\end{equation}

Recall that an {\it inversion}
of a permutation $\pi$ is a pair $(i,j)$
with $i<j$ and $\pi(i)>\pi(j)$. We denote the number of inversions
of $\pi$ by $inv(\pi)$.
It is well known that
$$
F_n(q):=\sum_{\pi\in S_n}q^{inv(\pi)}=[n]_q!.
$$

\begin{proposition}\label{inv-tableaux}
Let $W$ and $V$ be vectors, and $D$ and $E$ be matrices, satisfying
\begin{equation}
\braket W V =1, \qquad 
D\ket{V} = E\ket{V} \qquad
\bra{W}E=\bra{W}, \qquad 
DE=qED+D.
\end{equation}
Then the generating function $F_n(q)$ 
is equal to
$
\bra{W}D^n\ket{V}.
$
\end{proposition}

To prove this, we will use certain tableaux contained in shifted Young diagrams.

\begin{definition}
We equip the set $S:=\{(i,j) \in \mathbf{N} \times \mathbf{N} \ \vert \ i \leq j \}$
with the partial order:
$$(i,j) \geq (i',j') \text{ if and only if } i \geq i' \text{ and } j \geq j'.$$
A finite order ideal of $S$ is called a {\it shifted Young diagram}.
\end{definition}
See the left of Figure \ref{figInv1} for an example of a shifted Young diagram.

\begin{figure}[h!tp]
\centering
\psset{unit=0.4cm}\psset{linewidth=0.4mm}
\begin{pspicture}(0,0)(5,5)
\psline(5,0)(5,2)(6,2)(6,4)(7,4)(7,5)(0,5)(0,4)(1,4)(1,3)(2,3)(2,2)(3,2)(3,1)(4,1)(4,0)(5,0)
\psline(1,5)(1,4)
\psline(2,5)(2,3)
\psline(3,5)(3,2)
\psline(4,5)(4,1)
\psline(5,1)(4,1)
\psline(6,2)(3,2)
\psline(6,3)(2,3)
\psline(7,4)(1,4)
\psline(5,2)(5,5)
\psline(6,4)(6,5)
\rput(7.3,4.5){\tiny{D}}
\rput(6.6,3.8){\tiny{E}}
\rput(6.3,3.5){\tiny{D}}
\rput(6.3,2.5){\tiny{D}}
\rput(5.6,1.8){\tiny{E}}
\rput(5.3,1.5){\tiny{D}}
\rput(5.3,0.5){\tiny{D}}
\end{pspicture}
\hspace{2cm}
\begin{pspicture}(0,0)(5,5)
\psline(5,0)(5,5)(0,5)(0,4)(1,4)(1,3)(2,3)(2,2)(3,2)(3,1)(4,1)(4,0)(5,0)
\psline(1,5)(1,4)
\psline(2,5)(2,3)
\psline(3,5)(3,2)
\psline(4,5)(4,1)
\psline(5,1)(4,1)
\psline(5,2)(3,2)
\psline(5,3)(2,3)
\psline(5,4)(1,4)
\rput(0.5,4.5){1}
\rput(1.5,4.5){1}
\rput(2.5,4.5){0}
\rput(3.5,4.5){0}
\rput(4.5,4.5){0}
\rput(5.5,4.5){D}

\rput(1.5,3.5){1}
\rput(2.5,3.5){0}
\rput(3.5,3.5){1}
\rput(4.5,3.5){0}
\rput(5.5,3.5){D}

\rput(2.5,2.5){1}
\rput(3.5,2.5){1}
\rput(4.5,2.5){1}
\rput(5.5,2.5){D}

\rput(3.5,1.5){1}
\rput(4.5,1.5){1}
\rput(5.5,1.5){D}

\rput(4.5,0.5){1}
\rput(5.5,0.5){D}
\end{pspicture}
\hspace{2cm}
\begin{pspicture}(0,0)(5,5)
\psline(5,0)(5,5)(0,5)(0,4)(1,4)(1,3)(2,3)(2,2)(3,2)(3,1)(4,1)(4,0)(5,0)
\psline(1,5)(1,4)
\psline(2,5)(2,3)
\psline(3,5)(3,2)
\psline(4,5)(4,1)
\psline(5,1)(4,1)
\psline(5,2)(3,2)
\psline(5,3)(2,3)
\psline(5,4)(1,4)
\rput(0.5,4.5){1}
\rput(1.5,4.5){0}
\rput(2.5,4.5){1}
\rput(3.5,4.5){0}
\rput(4.5,4.5){0}
\rput(5.5,4.5){D}

\rput(1.5,3.5){1}
\rput(2.5,3.5){1}
\rput(3.5,3.5){1}
\rput(4.5,3.5){0}
\rput(5.5,3.5){D}

\rput(2.5,2.5){1}
\rput(3.5,2.5){1}
\rput(4.5,2.5){1}
\rput(5.5,2.5){D}

\rput(3.5,1.5){1}
\rput(4.5,1.5){1}
\rput(5.5,1.5){D}

\rput(4.5,0.5){1}
\rput(5.5,0.5){D}
\end{pspicture}
\caption{A shifted Young diagram and inversion tableaux}
\label{figInv1}
\end{figure}
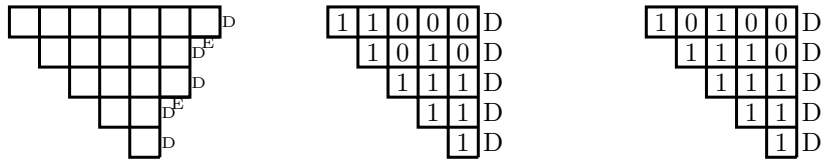

\begin{proof} 
Consider $0-1$ tableaux contained in shifted Young diagrams such that 
each cell of the main
diagonal contains a 1 and any cell that contains
a 0 has only 0's above it. 
See the diagrams of  Figure
\ref{figInv1}. 
Then there is a simple recurrence for such tableaux of a given shape,
that corresponds to the relations of Proposition \ref{inv-tableaux}.  Therefore 
$\bra{W}D^n\ket{V}$ 
enumerates such tableaux of staircase shape
$(n,n-1,n-2,\ldots,1)$.  We call such tableaux contained in a staircase 
shape {\it inversion tableaux}.

Label the columns of an inversion tableau from 1 to $n$ from right to left.
Then there is 
a bijection between inversion tableau and permutations, such that 
the number of superfluous
$1$'s in column $i$ corresponds to the number of inversions
of $\pi(i)$ (i.e. the number of $j>i$ such that
$\pi(i)>\pi(j)$). Indeed the sequence of the number of superfluous $1$'s in column $i=1,\ldots ,n$
is the inversion table of the permutation.
For example, the tableau 
in the middle of
 Figure \ref{figInv1} gives the inversion table $(2,2,0,1,0)$ and the permutation
$(3,4,1,5,2)$ and  the tableau 
on the right of Figure \ref{figInv1} gives the inversion table $(2,2,2,0,0)$ and the permutation
$(3,4,5,1,2)$.
This proves the proposition.

Note also that the following vectors and matrices satisfy the Ansatz:
$$
E_{i,i}=E_{i,i+1}=[i+1]_q, \qquad E_{i,j}=0 \ {\rm otherwise},
$$
$$
D_{i,i-1}=q^i[i]_q,\qquad D_{i,i}=q^i([i]_q+[i+1]_q), \qquad
\qquad D_{i,i+1}=q^i[i+1]_q,\qquad D_{i,j}=0\ {\rm otherwise},
$$
and $\bra{W}=(1,0,0,\ldots)$ and $\ket{V}=(1,0,0,\ldots )^T$.
Therefore $\bra{W}D^n\ket{V}$ is the generating function of weighted
 Motzkin paths
of length $n$ with weights $\lambda_i=q^{2i-1}[i+1]_q^2$ and $b_i=q^i([i]_q+[i+1]_q)$.
This is known to be the generating function of permutations counted
by inversions. See for example \cite{Bi93}.
\end{proof}

\subsection{Permutations and crossings}

\begin{definition}\cite{Cor07} \label{croreg}
A \emph{crossing} in a permutation $\pi$ is a pair $(i,j)$
such that
\begin{itemize}
\item $i< j\le \pi(i)<\pi(j)$ or
\item $i>j>\pi(i)>\pi(j)$.
\end{itemize}
\end{definition}

They appear in the combinatorial interpretation of the moments of Al-Salam-Chihara $q$-Laguerre 
polynomials \cite{KSZ08}. These are another $q$-analogue of Laguerre polynomials whose 
recurrence relation is:
\begin{equation}
  xL^*_n(x) = L^*_{n+1}(x) + (y[n+1]_q+[n]_q)L^*_n(x) + y[n]_ q^2L^*_{n-1}(x).
\end{equation}
See \cite{SS96} for some generalizations of these polynomials.

\begin{proposition}\cite{CoWi07a}
The generating function of permutations of size $n$ enumerated
according to their crossings is 
$$
\bra{W}(D+E)^n\ket{V}
$$
with 
\begin{equation}
\braket W V =1, \qquad
\bra{W}E=0, \qquad
D\ket{V}=\ket{V},\qquad
DE=qED+D+E.
\end{equation}
\end{proposition}

\begin{proof}
One can easily check that $
\bra{W}(D+E)^n\ket{V}
$ is the generating function of permutation tableaux
of length $n$ where $q$ counts the superfluous $1$'s, using the same arguments
as the ones developed in the introduction.
See \cite{CoWi07a}, and also
\cite{Jos10}, for proofs.
From \cite{StWi07} we know that permutation tableaux
of length $n$ with $j$ superfluous $1$'s are in bijection
with permutations of $[n]$ with $j$ crossings.

One can also prove  the result with a solution of the Matrix Ansatz.
Indeed the matrices $D$ and $E$ with entries given by 
\begin{eqnarray*}
D_{i,i}=[i+1]_q, &  D_{i,i+1}=[i+1]_q, & D_{i,j}=0\ {\rm otherwise},\\
E_{i,i}=[i]_q, &  E_{i,i-1}=[i]_q, & E_{i,j}=0\ {\rm otherwise},\\
\bra{W}=(1,0,\ldots),\ \ \ket{V}=(1,0,\ldots )^T
\end{eqnarray*}
are a solution.
Therefore $\bra{W}(D+E)^n\ket{V}$ is the generating function of weighted 
Motzkin paths
of length $n$ with weights $\lambda_i=[i]_q^2$ and $b_i=([i]_q+[i+1]_q)$.
This is known to also be the generating function of permutations counted
by crossings \cite{Cor07,SS96}.
\end{proof}

\subsection{Signed permutations}

A signed permutation of $\{1,\ldots ,n\}$ 
is a sequence of integers $\left(\pi(1),\ldots ,\pi(n)\right) $
such that $-n\le \pi(i)\le n$ and $\cup_{\ell=1}^n |\pi(\ell)|=\{1,\ldots ,n\}$.
For example
$\left(
2,3,7,-5,6,1,-4,8
\right)
$ is a signed permutation of $\{1,\ldots ,8\}$. Let $B_n$ be the set of signed permutations
of $\{1,\ldots ,n\}$.

We extend the definition of crossings of permutations to signed permutations:
\begin{definition}  \label{def_crossB}
A \emph{crossing} of a signed permutation $\pi=\left(\pi(1),\ldots ,\pi(n)\right)$ is a pair $(i,j)$ with $i,j>0$
such that
\begin{itemize}
\item $i<j\le \pi(i)<\pi(j)$ or
\item $-i<j\le -\pi(i)<\pi(j)$ or
\item $i>j>\pi(i)>\pi(j)$.
\end{itemize}
\end{definition}

\begin{remark} \label{remcroinv}
Note that if for all $i$, $\pi(i)>0$, 
then the crossings of the signed permutation
are the same as the crossings in Definition \ref{croreg} for usual permutations.
Moreover if for all $i$, $\pi(i)<0$, then the crossings of the signed permutation
are the same as the inversions of the permutation 
$(-\pi(1),\ldots ,-\pi(n))$.
This notion of crossing has the nice property that 
the number of signed permutations of $[n]$ with no crossings
is $\binom {2n}{n}$,
the Catalan number of type B,
just as the number of usual permutations of $[n]$ with no crossings
is the Catalan number of type A, see \cite{Wil05}.
\end{remark}

\begin{definition} \cite{LaWi08}
A {\em type B permutation tableau} of length $n$ is 
a filling with $0$'s and $1$'s of a shifted Young diagram of shape 
$\lambda=(\lambda_1,\lambda_2,\ldots ,\lambda_k)$, where
$\lambda_1\le n$, such that:
\begin{itemize}
\item Any 0 on the main diagonal has only 0's
above it.
\item There is no 0 with a 1 above it and a 1 to its left.
\item Each row has at least one 1.
\end{itemize}
\end{definition}

\begin{figure}[h!tp]
\centering
\psset{unit=0.4cm}\psset{linewidth=0.4mm}
\begin{pspicture}(0,0)(5,5)
\psline(5,0)(5,2)(6,2)(6,4)(7,4)(7,5)(0,5)(0,4)(1,4)(1,3)(2,3)(2,2)(3,2)(3,1)(4,1)(4,0)(5,0)
\psline(7,5)(8,5)
\psline(1,5)(1,4)
\psline(2,5)(2,3)
\psline(3,5)(3,2)
\psline(4,5)(4,1)
\psline(5,1)(4,1)
\psline(6,2)(3,2)
\psline(6,3)(2,3)
\psline(7,4)(1,4)
\psline(5,2)(5,5)
\psline(6,4)(6,5)
\rput(0.5,4.5){1}
\rput(1.5,4.5){0}
\rput(2.5,4.5){1}
\rput(3.5,4.5){0}
\rput(4.5,4.5){0}
\rput(5.5,4.5){1}
\rput(6.5,4.5){0}
\rput(1.5,3.5){1}
\rput(2.5,3.5){1}
\rput(3.5,3.5){0}
\rput(4.5,3.5){0}
\rput(5.5,3.5){1}
\rput(2.5,2.5){1}
\rput(3.5,2.5){0}
\rput(4.5,2.5){0}
\rput(5.5,2.5){1}
\rput(3.5,1.5){0}
\rput(4.5,1.5){1}
\rput(4.5,0.5){1}
\rput(5.3,0.5){\tiny{D}}
\rput(5.3,1.5){\tiny{D}}
\rput(5.7,1.8){\tiny{E}}
\rput(6.3,2.5){\tiny{D}}
\rput(6.3,3.5){\tiny{D}}
\rput(6.7,3.8){\tiny{E}}
\rput(7.3,4.5){\tiny{D}}
\rput(7.7,4.9){\tiny{E}}

\end{pspicture}\hspace{2cm}
\caption{\label{exaAB} Permutation tableau of type B}
\end{figure}
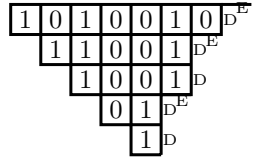

These tableaux are in bijection with signed permutations,
and are related to a cell decomposition of the totally non-negative
part of a type B Grassmannian \cite{LaWi08}.

An example is given in Figure \ref{exaAB}.
Note that we allow rightmost columns of length $0$.
We will label the southeast border of a tableau
with $D$'s and $E$'s, writing a $D$ (resp. $E$)
whenever we go along a vertical (resp. horizontal) step.
This word in $D$'s and $E$'s encodes the {\it shape} of the tableau.

Let $B_n(q,r)$ be the generating function of 
permutation tableaux of type B of length $n$,
where $q$ counts
the superfluous $1$'s and $r$ counts the $1$'s on the main diagonal. 
\begin{proposition}
The generating function $B_n(q,r)$ is equal to
$$
\bra{W}(D+E)^n\ket{V}
$$
where
\begin{equation}
\braket WV =1, \qquad 
\bra{W}E=\bra{W}, \qquad
D\ket{V}=rE\ket{V}, \qquad
DE=qED+E+D.
\end{equation}
\end{proposition}

\begin{proof}
We illustrate a recursion for the generating function for 
type B permutation tableaux of a fixed shape, which mirrors the relations above.
First, the generating function of the tableaux of length 0 is 1 : this
corresponds to the fact that 
$\braket WV =1$.
Second, a bottommost row
of length one must contain a $1$ on the main diagonal.  This row
can be deleted, which corresponds to $D\ket{V}=rE\ket{V}$.
Third, any  column of length $0$ at the right can be deleted, which 
corresponds to $\bra{W} E=\bra{W}$.
Finally each corner box (which corresponds to the subword $DE$ 
in the word encoding the shape of the tableau) contains either:
\begin{itemize}
\item a superfluous $1$, so this box 
can be deleted (leaving a tableau whose shape has the $DE$ replaced by $ED$,
and whose weight is equal to the old weight divided by $q$)
\item the unique $1$ in its row, and consequently the row and column containing 
the main diagonal cell to its left can be deleted
($+D$)
\item a 0, and consequently its column  can be deleted ($+E)$.
\end{itemize}
See Figure \ref{decomp} for an example.
\end{proof}

\begin{figure}[h!tp]
\centering
\psset{unit=0.4cm}\psset{linewidth=0.4mm}
\begin{pspicture}(0,-1)(5,5)
\psline(5,0)(5,2)(6,2)(6,4)(7,4)(7,5)(0,5)(0,4)(1,4)(1,3)(2,3)(2,2)(3,2)(3,1)(4,1)(4,0)(5,0)
\psline(1,5)(1,4)
\psline(2,5)(2,3)
\psline(3,5)(3,2)
\psline(4,5)(4,1)
\psline(5,1)(4,1)
\psline(6,2)(3,2)
\psline(6,3)(2,3)
\psline(7,4)(1,4)
\psline(5,2)(5,5)
\psline(6,4)(6,5)
\rput(5.5,2.5){*}
\rput(5.3,0.5){\tiny{D}}
\rput(5.3,1.5){\tiny{D}}
\rput(5.7,1.8){\tiny{E}}
\rput(6.3,2.5){\tiny{D}}
\rput(6.3,3.5){\tiny{D}}
\rput(6.7,3.8){\tiny{E}}
\rput(7.3,4.5){\tiny{D}}
\rput(3,-1){$DE$}
\end{pspicture}\hspace{2cm}
\begin{pspicture}(0,-1)(5,5)
\rput(-1,3){$=$}
\psline(5,0)(5,2)(6,2)(6,4)(7,4)(7,5)(0,5)(0,4)(1,4)(1,3)(2,3)(2,2)(3,2)(3,1)(4,1)(4,0)(5,0)
\psline(1,5)(1,4)
\psline(2,5)(2,3)
\psline(3,5)(3,2)
\psline(4,5)(4,1)
\psline(5,1)(4,1)
\psline(6,2)(3,2)
\psline(6,3)(2,3)
\psline(7,4)(1,4)
\psline(5,2)(5,5)
\psline(6,4)(6,5)
\rput(5.5,2.5){1}
\rput(3,-1){$=qED$}
\end{pspicture}\hspace{2cm}
\begin{pspicture}(0,-1)(5,5)
\psline(5,0)(5,2)(6,2)(6,4)(7,4)(7,5)(0,5)(0,4)(1,4)(1,3)(2,3)(2,2)(3,2)(3,1)(4,1)(4,0)(5,0)
\psline(1,5)(1,4)
\psline(2,5)(2,3)
\psline(3,5)(3,2)
\psline(4,5)(4,1)
\psline(5,1)(4,1)
\psline(6,2)(3,2)
\psline(6,3)(2,3)
\psline(7,4)(1,4)
\psline(5,2)(5,5)
\psline(6,4)(6,5)
\rput(5.5,2.5){1}
\rput(4.5,2.5){0}
\rput(3.5,2.5){0}
\rput(2.5,2.5){0}
\rput(2.5,3.5){0}
\rput(2.5,4.5){0}
\rput(3,-1){$+E$}
\end{pspicture}\hspace{2cm}
\begin{pspicture}(0,-1)(5,5)
\psline(5,0)(5,2)(6,2)(6,4)(7,4)(7,5)(0,5)(0,4)(1,4)(1,3)(2,3)(2,2)(3,2)(3,1)(4,1)(4,0)(5,0)
\psline(1,5)(1,4)
\psline(2,5)(2,3)
\psline(3,5)(3,2)
\psline(4,5)(4,1)
\psline(5,1)(4,1)
\psline(6,2)(3,2)
\psline(6,3)(2,3)
\psline(7,4)(1,4)
\psline(5,2)(5,5)
\psline(6,4)(6,5)
\rput(5.5,2.5){0}
\rput(5.5,3.5){0}
\rput(5.5,4.5){0}
\rput(3,-1){$+D$}
\end{pspicture}\hspace{2cm}
\caption{Decomposition of permutation tableaux of type B }
\label{decomp}
\end{figure}
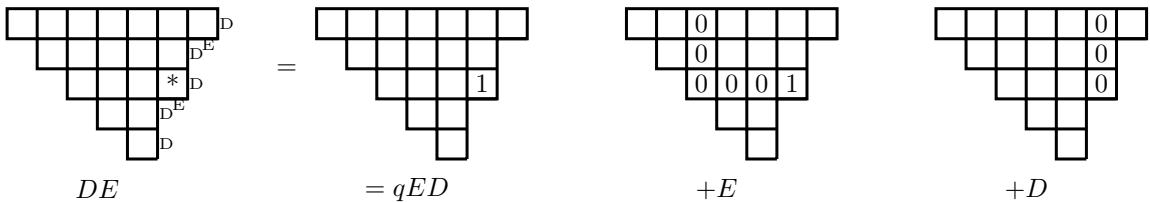

\begin{theorem}
There exists a one-to-one correspondence between
 signed permutations in $B_n$ with
$j$  crossings and $\ell$ minus signs and
permutation tableaux of type $B$ of length $n$ with $j$ superfluous
$1$'s and $\ell$ $1$'s on the main diagonal.
\end{theorem}
\begin{proof}
This result is analogous to a result proved for usual permutations
and type A permutation tableaux \cite{Cor07, StWi07}, and we can prove
it using an argument analogous to that given in \cite{StWi07}.
Indeed, we use the bijection from \cite[Section 10]{LaWi08}, 
which in turn was an adaptation of the bijection from \cite{StWi07}.
We then apply the same arguments that were used in Section 3 of
\cite{StWi07} to show how the statistics translate under this
bijection. 

One can also prove the preceding result with a solution of the Matrix Ansatz.
Indeed
$$
D_{i,i}=D_{i,i+1}=[i+1]_q,\qquad D_{i,j}=0\ {\rm otherwise};
$$
$$
E_{i,i-1}=[i]_q(1+rq^i),\qquad E_{i,i}=[i]_q(1+rq^i)+rq^i[i+1]_q, \qquad
\qquad E_{i,i+1}=rq^i[i+1]_q, \qquad E_{i,j}=0\ {\rm otherwise},
$$
$$\bar{W}=(1,0,\ldots),\qquad \ket{V}=(1,0,\ldots )^T,
$$
is a solution.
Therefore $\bra{W}(D+E)^n\ket{V}$ is the generating function of weighted 
Motzkin paths
of length $n$ with weights $\lambda_i=[i+1]_q^2(1+rq^i)(1+rq^{i+1})$ 
and $b_i=(1+rq^i)([i]_q+[i+1]_q)$.
On can adapt the proofs in  \cite{Cor07,SS96} to show that this is indeed 
the generating function of signed permutation in $B_n$ according to 
crossings \cite{CJKW10}. 
\end{proof}

From the weights on the Motzkin paths, we see that the generating function for signed 
permutations of size $n$ where $q$ counts the number of crossings, is the $n$th moment
of the orthogonal sequence $P_n(x;r\mid q)$ or simply $P_n(x)$ defined by the recurrence:
\begin{equation}
  xP_n(x) = P_{n+1}(x) + \big([n+1]_q + [n]_q\big)(1+rq^n) P_n(x)
 + (1+rq^n)(1+rq^{n-1})[n]_q^2 P_{n-1}(x),
\end{equation}
with $P_{-1}(x)=0$ and $P_0(x)=1$.
We have the following two special cases for these orthogonal polynomials (which are the counterpart
of combinatorial properties given in Remark \ref{remcroinv}).
\begin{itemize}
\item When $r=0$, we recover the recurrence relation for the Al-Salam-Chihara $q$-Laguerre
      polynomials $L_n^*(x)$,
\item When we keep only the terms of maximal degree in $r$, we recover the recurrence relation
      for the little $q$-Jacobi Laguerre polynomials $L_n(x)$.
\end{itemize}
These $P_n(x)$ can be linked with some classical polynomials called continuous dual 
$q$-Hahn polynomials in \cite{KoSw98}, denoted by $R_n(x;a,b,c\mid q)$ or simply $R_n(x)$.
Their recurrence relation is
\begin{equation} \label{rec_cdqh}
  2xR_n(x) = R_{n+1}(x) + \left(a+\tfrac 1a - (A_n + C_n)\right)R_n(x) + A_{n-1}C_nR_{n-1}(x),
\end{equation}
where
\begin{equation}
 A_n =  \tfrac 1a (1-abq^n)(1-acq^n) \qquad \hbox{ and } \qquad C_n = a (1-q^n)(1-bcq^{n-1}).
\end{equation}
In the case where $a=1$, $b=-r$ and $c=q$, one has
\begin{equation}
 A_n +C_n =  -(1-q)\big([n]_q+[n+1]_q)(1+rq^n) \quad \hbox{ and } 
 \quad A_{n-1}C_n =  (1-q)^2 (1+rq^{n})(1+rq^{n-1}) [n]_q^2 .
\end{equation}
Hence,
\begin{equation}
   P_n(x;r\mid q) = \tfrac{1}{(q-1)^n} R_n\left( x\tfrac{q-1}{2} + 1 ;1,-r,q\mid q\right),
\end{equation}
and we can see $P_n(x)$ as a special case of these continuous dual $q$-Hahn polynomials.

%

\section{Genocchi numbers}
\label{sec:gen}

Dumont and Foata introduced in \cite{DuFo76} a three-parameter generalization 
$\{f_n(a,b,c)\}_{n\geq1}$ 
of the Genocchi numbers,
defined by $f_1(a,b,c)=1$ and
\begin{equation} \label{df_rec}
  f_{n+1}(a,b,c) = (a+b)(a+c) f_{n}(a+1,b,c) - a^2 f_{n}(a,b,c).
\end{equation}

When the parameters are set to $1$, $f_n(1,1,1)$ 
is the Genocchi number $G_{2n+2}$ defined by
\begin{equation}
\sum_{n=1}^{\infty} G_{2n} \frac{ x^{2n} }{ (2n)! } = x \cdot \hbox{tan}\left(\frac x2\right).
\end{equation}

It is straightforward to see that 
$f_n$ has non-negative coefficients, but it is more difficult to 
prove that $f_n$ is symmetric in $a$, $b$, and $c$ \cite{DuFo76}.
And while several 
combinatorial interpretations of $f_n(a,b,c)$ have been given 
\cite{Dum74,DuFo76,Han96,Vie81}, none of them readily exhibit 
the symmetry in $a, b, c$. 

The goal of this section is to give a new combinatorial interpretation
of $f_n(a,b,c)$ in terms of {\it alternative tableaux} \cite{Vie08}, and to use this 
combinatorial interpretation to give a simple proof of the symmetry
in $a, b$ and $c$. This is done by using the link with moments 
of continuous dual Hahn polynomials, and the Matrix Ansatz for alternative
tableaux.


The continuous dual Hahn polynomials $S_n(x)$ are defined by the recurrence :
\begin{equation} \label{rec_cdh}
  xS_n(x) = S_{n+1}(x) + (A_n+C_n-a^2)S_n(x) 
     + A_{n-1}C_nS_{n-1}(x),
\end{equation}
where:
\begin{equation}
 A_n = (n+a+b)(n+a+c) \qquad\hbox{ and }\qquad C_n= n(n+b+c-1),
\end{equation}
and $S_{-1}(x)=0$, $S_0(x)=1$. (Up to a sign they are the same as the ones defined in 
\cite{KoSw98}.)
Zeng \cite[Corollaire 2]{Zen96} has given
a continued fraction for the generating function $\sum_{n\geq0} f_n(a,b,c)x^n$. 
This was conjectured by Dumont and also proved by Randrianarivony, see details in \cite{Zen96}.
From this 
continued fraction and the recurrence (\ref{rec_cdh}) we see that $f_n(a,b,c)$ is the $n$th moment 
$\mu_n^s$ of the orthogonal sequence $\{S_k(x)\}_{k\geq0}$.

\medskip

Thus we have $f_n(a,b,c)=\bra{W}M^n\ket{V}$ where the matrix $M$ has coefficients:
\begin{equation}
M_{i,i} = A_i + C_i - a^2, \qquad M_{i+1,i} = A_i, \qquad M_{i-1,i} = C_i.
\end{equation}

Using the bijection of Corteel and Nadeau \cite{CoNa09}, it can be checked that permutation
tableaux of staircase shape are linked to {\it Dumont permutation of the first kind}
\cite{Dum74}, which were introduced to give a combinatorial interpretation of Genocchi numbers.
These permutations are the $\sigma\in\mathfrak{S}_{2n}$ such that $\sigma(i)>\sigma(i+1)$
if and only if $\sigma(i)$ is even for any $1\leq i \leq 2n$ (with the convention that 
$\sigma(2n+1)=2n+1$). Via the bijection, the shape of the tableau is obtained by examining the
values of the descent in the permutation so that the result follow easily.
Thus, Genocchi numbers counts the permutation tableaux of staircase shape. 
Alternatively, we could use the bijection of Steingr\'imsson and Williams \cite{StWi07}, which links 
staircase permutation tableaux to {\it Dumont permutation of the second kind} \cite{Dum74};
indeed, these permutations are the $\sigma\in\mathfrak{S}_{2n}$ such that $\sigma(i)<i$ if and only
if $i$ is even, and via the bijection the shape of the tableau is obtained by examining the 
weak exceedances of the permutations.

So it is natural to compare this matrix $M$ with the product $DE$ of the matrices $D$ and 
$E$ of the PASEP Matrix Ansatz defined in \eqref{useafter}. The limit when $q=1$ of these matrices are 
well-defined, and a straightforward computation show that $DE+(c-1)(D+E)$ 
is equal to $M$, under the condition that $a=\beta^{-1}$ and $b=\alpha^{-1}$. So, knowing that 
$f_n(a,b,c)$ is symmetric, we have:
\[
f_n(a,b,c)=\bra{W}M^n\ket{V}=\bra{W}(DE+(c-1)(D+E))^n\ket{V} = \bra{W}(ED+cD+cE))^n\ket{V}.
\]
The last equality is obtained by using the fact that $DE=ED+D+E$.
We can derive a new combinatorial interpretation of $f_n(a,b,c)$ in terms of permutation tableaux.

\smallskip

We will also use {\it alternative tableaux}, which are slightly different objects.
As remarked in \cite{CoNa09}, all the entries of a permutation tableau
can be recovered if one knows the position of the topmost 1's
and the rightmost restricted 0's. Viennot \cite{Vie08,Na09} took this
a step further and defined {\it alternative tableaux}, which are
partial fillings of Young diagrams with $\alpha$s and $\beta$s
such that any cell to the left of a $\beta$ (resp. above an $\alpha$)
is empty. There is a direct bijection from permutation tableaux
of length $n+1$ to alternative tableaux of length $n$
(essentially one replaces topmost $1$'s with $\alpha$s and 
rightmost restricted $0$'s with $\beta$s and makes all other boxes
empty and deletes the first row). 
Alternative tableaux are interesting in this context, because
they are more symmetric than permutation tableaux, and consequently more
adequate to explain the symmetry of the polynomials $f_n(a,b,c)$.

\begin{theorem} The polynomial $(\alpha\beta\gamma)^nf_n(\alpha^{-1},\beta^{-1},\gamma^{-1})$ counts 
the staircase alternative tableaux of length $2n$, where the parameters $\alpha$, 
$\beta$, $\gamma$ follow these statistics:
\begin{itemize}
\item the number of cells containing a $\alpha$,
\item the number of cells containing a $\beta$, 
\item the number of corners which does not contains a $\alpha$ or a $\beta$.
\end{itemize}
Equivalently, $(\alpha\beta\gamma)^nf_n(\alpha^{-1},\beta^{-1},\gamma^{-1})$ counts the 
staircase permutation tableaux of length $2n+2$, where the parameters $\alpha$, $\beta$, $\gamma$ 
follow these statistics:
\begin{itemize}
\item the number of $0$'s in the first row,
\item the number of restricted rows,
\item the number of corner which contains a superfluous $1$.
\end{itemize}
\end{theorem}

\begin{proof}
When $\gamma=1$, we can use the results from \cite{CoWi07a} which have been recalled in the 
introduction of this article, and it follows
that $\bra{W}(DE)^n\ket{V}$ counts the staircase permutation tableaux of
length $2n+2$, where the parameters $\alpha^{-1}$, $\beta^{-1}$ counts the number of 1's in the 
first row and the number of unrestricted row.
When we replace $DE$ with $ED+\gamma^{-1}D+\gamma^{-1}E$, in the recurrence relation for permutation 
tableaux we see that $\gamma^{-1}$ will count the corners containing a restricted 0 or a topmost 1.
So in $(\alpha\beta\gamma)^nf_n(\alpha^{-1},\beta^{-1},\gamma^{-1})$, the parameters $\alpha$,
$\beta$, and $\gamma$ counts the complementary statistics which are given in the theorem.
\end{proof}

\medskip

In this combinatorial interpretation, the symmetry in $\alpha$ and $\beta$ is apparent, because 
we can transpose the tableaux to exchange the two parameters. The symmetry in $\gamma$ and $\beta$ 
is obvious from the recurrence relation \eqref{df_rec}, so this implies the full symmetry of the three
parameters. The symmetry in $\gamma$ and $\beta$ 
can also be proved by an explicit involution \cite{Jos11}. 

We can ask if there is a direct bijection between staircase alternative tableaux and other
known combinatorial interpretations of $f_n(a,b,c)$. In particular, Viennot \cite{Vie81}
gives a combinatorial interpretation which has also an apparent symmetry exchanging two 
parameters, which can be defined as follows. We consider pairs $(f,g)$ of maps from $[n]$
to $[n]$ such that $f(i)\geq i$ and $g(i)\geq i$ for any $i$, and such that for any 
$1\leq j\leq n$ there is at least an $i$ such that $f(i)=j$ or $g(i)=j$.
The three statistics are:
\begin{itemize}
\item $u(f,g)$ is the number of $i$ such that $f(i)=i$ ,
\item $v(f,g)$ is the number of $i$ such that $g(i)=i$ ,
\item $w(f,g)$ is the number of $i$ such that $f(i)=n$ 
plus the the number of $j$ such that $g(j)=n$.
\end{itemize}
The we have $f_n(a,b,c)=\sum_{(f,g)}a^{u(f,g)-1}b^{v(f,g)-1}c^{w(f,g)-2}$.
There is an obvious symmetry in $a$ and $b$ but it seems to be different 
from the symmetry which is apparent on
the permutation tableaux.

\medskip

If we examine the general case where $q\neq 1$, it appears that $\bra{W}(qED+cD+cE)^n\ket{V}$ is not
symmetric in $\alpha$ and $\gamma$. To obtain a refinement of $f_n(a,b,c)$ or 
$(\alpha\beta\gamma)^nf_n(\alpha^{-1},\beta^{-1},\gamma^{-1})$ it might be interesting to consider
the continuous dual $q$-Hahn polynomials $R_n(x;a,b,c\mid q)$, defined in \eqref{rec_cdqh}.
Let $g_n(a,b,c,q)$ be the $n$th moment of the polynomials $R_n(\frac x2-1;a,b,c\mid q)$. Let 
\[
  \tilde a = (1-q)a-1, \qquad \tilde b = (1-q)b-1, \qquad \hbox{and} \quad \tilde c = (1-q)c-1.
\]
Then we have that $g_n(\tilde a,\tilde b,\tilde c,q)(1-q)^{-2n}$ is a polynomial, symmetric 
in $a$, $b$ and $c$, which specializes to $f_n(a,b,c)$ when $q=1$. This just confirms that 
these $q$-Hahn polynomials are a $q$-analog of the Hahn polynomials. Unfortunately the 
polynomial $g_n(\tilde a,\tilde b,\tilde c,q)(1-q)^{-2n}$ contains negative terms. However 
we have the following result when one parameter is set to one, for example $a=1$.

\begin{theorem} The moment $g_n(-q,\tilde b,\tilde c,q)(1-q)^{-2n}$ is the generating function 
for staircase permutation tableaux of length $2n+2$, where $q$ counts the superfluous $1$'s, $b$ 
counts the number of $1$'s in the first row except one of them, and $c$ counts the number of 
unrestricted rows except the first row.
\end{theorem}

\begin{proof}
This generating function of permutation tableaux is $\bra{W}(DE)^n\ket{V}$ where $D$ and $E$ 
are the solution of the PASEP Matrix Ansatz. By writing the product explicitely we can check 
that $M=DE$ has coefficients:
\[
   (1-q)M_{i,i} = (1-\tilde\alpha q^i)(1- \tilde \beta q^i) + (1-\tilde\alpha\beta q^{i})(1-q^{i+1}),
\]
\[
   (1-q)M_{i+1,i} = (1-q^{i+1})(1-\tilde\beta q^{i+1} ),
\]
\[
   \qquad (1-q)M_{i,i+1} = (1-\tilde\alpha\tilde \beta q^i)(1- \tilde \alpha q^{i+1}). 
\]
Hence $\bra{W}(DE)^n\ket{V}$ is the $n$th moment of an orthogonal sequence where the three-term
relation can be derived from the coefficients of the matrix $M$, and by comparing it with the 
coefficients of the recurrence (\ref{rec_cdqh}) we can derive the result.
\end{proof}




\begin{thebibliography}{999999999}

\bibitem[ASC65]{ASC65}
 W.A. Al-Salam et L. Carlitz, Some orthogonal $q$-polynomials, 
 Math. Nachr. 30 (1965) 47--61. 

\bibitem[Bi93]{Bi93}  P. Biane, Permutations suivant le type d'exc\'edance 
 et le nombre d'inversions et interpr\'etation combinatoire 
 d'une fraction continue de Heine, 
 European. J. Combin. 14 (1993) 277--284.

\bibitem[BECE00]{BECE}
 R.A. Blythe, M.R. Evans, F. Colaiori, F.H.L. Essler,
 Exact solution of a partially asymmetric exclusion model using a deformed
 oscillator algebra, J. Phys. A: Math. Gen. 33 (2000) 2313--2332.

\bibitem[BCEPR06]{BCEPR} 
 R. Brak, S. Corteel, J. Essam, R. Parviainen, A. Rechnitzer,
 A combinatorial derivation of the PASEP stationary state,
 Electron. J. Combin.  13  (2006) R108.


\bibitem[Bur07]{Bur07}
 A. Burstein,  On some properties of permutation tableaux,
 Ann. Combin. 11 (2007) 355--368.

\bibitem[Cor07]{Cor07}
S. Corteel, Crossings and alignments of permutations, Adv. in Appl. Math.
38 (2007) 149--163.



\bibitem[CJKW10]{CJKW10} S. Corteel, M. Josuat-Verg\`es, J.S. Kim and L.K. Williams,
The combinatorics of permutation tableaux of type B, in preparation 2010.

\bibitem[CoNa09]{CoNa09}
S. Corteel, P. Nadeau, Bijections for permutation tableaux,
Eur. J. of Comb. 30 (2009) 295--310.

\bibitem[CoWi07a]{CoWi07a}
 S. Corteel, L.K. Williams, Tableaux combinatorics for the asymmetric 
 exclusion process, Adv. in Appl. Math. 39 (2007) 293--310.


\bibitem[CoWi10]{CoWi10}
 S. Corteel, L. K. Williams, Tableaux Combinatorics for the Asymmetric Exclusion Process and Askey-Wilson polynomials, Proc. Natl. Acad. Sci., 
published online ahead of print March 26, 2010, doi:10.1073/pnas.0909915107.

\bibitem[DEHP93]{DEHP93}
 B. Derrida, M. Evans, V. Hakim, V. Pasquier, Exact solution of a 1D asymmetric 
 exclusion model using a matrix formulation, J. Phys. A: Math. Gen. 26 (1993)
 1493--1517.

\bibitem[Dum74]{Dum74}
 D. Dumont, Interpr\'etations combinatoires des nombres de Genocchi, 
 Duke Math. J. 41 (1974) 305--318.

\bibitem[DuFo76]{DuFo76}
 D. Dumont, D. Foata, Une propri\'et\'e de sym\'etrie des nombres de Genocchi, 
 Bull. Soc. Math. France 104 (1976) 433--451

\bibitem[Fan92]{Fannes} M. Fannes, B. Nachtergaele, R. Werner,
Finitely correlated states on quantum spin chains, 
Commun. Math. Phys. 144, 1992.

\bibitem[Fla82]{Fla82}
 P. Flajolet, Combinatorial aspects of continued fractions, Discrete Math. 41 (1982) 145--153.



\bibitem[HaNa83]{HakimNadal} V. Hakim, J. Nadal, Exact results for 
2D directed animals on a strip of finite width, J. Phys. A: Math. Gen. 16
(1983) L213-L218.

\bibitem[Han96]{Han96}
 G.-N. Han, Sym\'etries trivari\'ees sur les nombres de Genocchi,
 European. J. Combin. 17 (1996) 397--407.
 
 \bibitem[ISV87]{ISV87}
 M.E.H. Ismail, D. Stanton, X.G. Viennot, The combinatorics of $q$-Hermite polynomials and the
 Askey-Wilson integral, European J. Combin. 8 (1987) 379--392.



\bibitem[Jos08b]{Jos08b}
 M. Josuat-Verg\`es, Rook placements in Young diagrams and permutation enumeration,
 arXiv:0811.0524v2.

\bibitem[Jos09a]{Jos09a}
 M. Josuat-Verg\`es, A $q$-enumeration of alternating permutations, to appear in European J. Combin. (2010).

\bibitem[Jos09b]{Jos09b}
 M. Josuat-Verg\`es, Combinatorics of the 3-parameter PASEP partition function,
 ArXiv:0912.1279v2.

\bibitem[Jos10]{Jos10}
 M. Josuat-Verg\`es, \'Enum\'eration de tableaux et de chemins, moments de polyn\^omes orthogonaux,
 PhD thesis, Universit\'e Paris-Sud, Orsay, 2010.

\bibitem[Jos11]{Jos11}
 M. Josuat-Verg\`es, Generalized Dumont-Foata polynomials and alternative tableaux, in preparation.

 

\bibitem[KSZ08]{KSZ08}
 A. Kasraoui, D. Stanton, J. Zeng, The combinatorics of Al-Salam-Chihara 
 $q$-Laguerre polynomials, arXiv.org:0810.3232v1.

\bibitem[KaZe06]{KaZe06}
 A. Kasraoui, J. Zeng, Distribution of crossings, nestings and alignments of
 two edges in matchings and partitions, Electron. J. Combin. 13 (2006) Article R33.

\bibitem[KSZ06]{KSZ06}
 D. Kim, D. Stanton, J. Zeng,
 The combinatorics of the Al-Salam-Chihara $q$-charlier polynomials,
 S\'eminaire Lotharingien de Combinatoire 54 (2006) Article B54i.

\bibitem[Kl91]{Klumper} 
A. Klumper, A. Schadschneider, J. Zittartz, Equivalence and solution
of anisotropic spin-1 models and generalized t-J fermion models
in one dimension. J. Phys. A: Math. Gen 24 (1991) L955-L959.

\bibitem[KoSw98]{KoSw98}
 R. Koekoek, R.F. Swarttouw, The Askey-scheme of hypergeometric orthogonal 
 polynomials and its $q$-analogue, Delft University of Technology, 
 Report no. 98--17 (1998).


\bibitem[LaWi08]{LaWi08}
 T. Lam, L.K. Williams, Total positivity for cominuscule Grassmannians,
 New York J. Math. 14 (2008) 53--99. 

\bibitem[Ler90]{Ler90}
 P. Leroux, Reduced matrices and $q$-log concavity properties of $q$-Stirling numbers, 
 J. Combin. Theory Ser. A 54 (1990) 64--84.


\bibitem[MSS07]{MSS07}
 T. Mansour, M. Schork, S. Severini, Wick's theorem for $q$-deformed boson operators, 
 J. Phys. A: Math. Theor. 40 (2007) 8393--8401.

\bibitem[MSW95]{MSW95}
 A. de M\'edicis, D. Stanton, D. White, The combinatorics of $q$-Charlier polynomials
 J. Combin. Theory Ser. A 69 (1995) 87--114.

\bibitem[MeVi94]{MeVi94}
 A. de M\'edicis, X. G. Viennot, Moments of Laguerre's $q$-polynomials and Foata-Zeilberger's 
 bijection, Adv. in Appl. Math. 15 (1994) 262--304.

\bibitem[Na09]{Na09} P. Nadeau,  The structure of alternative tableaux, preprint (2009), arXiv:0908.4050.

\bibitem[NTW08]{NTW08}
 J.-C. Novelli, J.-Y. Thibon, L. K. Williams, 
 Combinatorial Hopf algebras, noncommutative Hall-Littlewood functions, and permutation tableaux, to appear in Adv. Math.

\bibitem[Pen95]{Pen95} 
 J.-G. Penaud,Une preuve bijective d'une formule de Touchard-Riordan, 
Discrete Math. 139 (1995) 347--360.

\bibitem[Pos06]{Pos06}
 A. Postnikov, Total positivity, grassmannians, and networks,
 preprint 2006, arXiv:math/0609764v1.
 


\bibitem[SS96]{SS96} R. Simion, D. Stanton, 
Octabasic Laguerre polynomials and permutation statistics,
J. Comput. Appl. Math. 68  (1996) 297-329. 


\bibitem[Sa99]{Sa99} T. Sasamoto, 
One-dimensional partially asymmetric simple exclusion process with open boundaries: 
Orthogonal polynomials approach. J. Phys. A, Math. Gen. 32, No.41, 7109-7131 (1999).

\bibitem[StWi07]{StWi07} 
 E. Steingr\'imsson, L.K. Williams, Permutation tableaux and permutation patterns, 
 J. Combin. Theory Ser. A 114 (2007) 211--234.


\bibitem[Var05]{Var05}
 A. Varvak, Rook numbers and the normal ordering problem, 
 J. Combin. Theory Ser. A 112 (2005) 292--307.

\bibitem[Vie08]{Vie08}
 X. G. Viennot, Alternative tableaux, permutations and partially asymmetric 
 exclusion process, talk in Isaac Newton Institute, April 2008. 

 \url{http://www.newton.ac.uk/webseminars/pg+ws/2008/csm/csmw04/}

\bibitem[Vie81]{Vie81}
 X.G. Viennot, Interpr\'etations combinatoires des nombres d'Euler and de Genocchi,
 s\'eminaire de th\'eorie des nombres de l'universit\'e Bordeaux I, Bordeaux, 1981.

\bibitem[Vie88]{Vie88}
 X.G. Viennot, Une th\'eorie combinatoire des polyn\^omes 
 orthogonaux, Notes de cours, UQ\`AM, Montr\'eal, 1988.
\url{http://web.mac.com/xgviennot/Xavier_Viennot/livres.html}.


\bibitem[Wil05]{Wil05}
 L.K. Williams,  Enumeration of totally positive Grassmann cells,
 Adv. Math. 190 (2005) 319--342.

\bibitem[Zen96]{Zen96}
  J. Zeng, Sur quelques propri\'et\'es de sym\'etrie des nombres de Genocchi, 
  Discrete Math. 153 (1996) 319--333. 

\end{thebibliography}
\end{document}